\documentclass[11pt]{article}
\usepackage{amsmath, amscd, amssymb, latexsym, epsfig, color, amsthm, multicol, authblk}
\usepackage[all]{xy}
\usepackage{tikz}
\usepackage{verbatim, hyperref}
\usepackage{float}

\numberwithin{equation}{section}

\newtheorem{theorem}{Theorem}[section]
\newtheorem{proposition}[theorem]{Proposition}

\newtheorem{conjecture}[theorem]{Conjecture}
\newtheorem{lemma}[theorem]{Lemma}

\theoremstyle{definition}
\newtheorem{definition}[theorem]{Definition}
\newtheorem{example}[theorem]{Example}

\newtheorem{remark}[theorem]{Remark}

\DeclareMathOperator\lk{\mathrm{lk}}
\DeclareMathOperator\st{\mathrm{st}}

\DeclareMathOperator{\conv}{\mathrm{conv}}

\DeclareMathOperator{\dist}{\mathrm{dist}}

\newcommand{\field}{{\bf k}}

\newcommand{\R}{{\mathbb R}}

\newcommand{\Z}{{\mathbb Z}}

\newcommand{\Sp}{\mathbb{S}}
\newcommand{\twprod}{\mathbin{\mathchoice%
		{\ooalign{\raise1.15ex\hbox{$\scriptstyle\sim$}\cr\hidewidth$\times$\hidewidth\cr}}%
		{\ooalign{\raise1.15ex\hbox{$\scriptstyle\sim$}\cr\hidewidth$\times$\hidewidth\cr}}%
		{\ooalign{\raise.85ex\hbox{$\scriptscriptstyle\sim$}\cr\hidewidth$\scriptstyle\times$\hidewidth\cr}}%
		{\ooalign{\raise.65ex\hbox{$\scriptscriptstyle\sim$}\cr\hidewidth$\scriptscriptstyle\times$\hidewidth\cr}}%
}}

\title{Minimal flag triangulations of lower-dimensional manifolds}
\author[1]{Christin Bibby \thanks{bibby@umich.edu}}
\author[1]{Andrew Odesky\thanks{aodesky@umich.edu}}
\author[1]{Mengmeng Wang\thanks{mengmw@umich.edu}}
\author[1]{Shuyang Wang\thanks{shuywang@umich.edu}}
\author[1]{Ziyi Zhang\thanks{zizh@umich.edu}}
\author[1]{Hailun Zheng\thanks{hailunz@umich.edu}}
\affil{University of Michigan\protect \\ Department of Mathematics\protect \\
University of Michigan\protect \\Ann Arbor, Michigan 48109-1043, USA}
\begin{document}
\maketitle
\begin{abstract}
	We prove the following results on flag triangulations of 2- and 3-manifolds. In dimension 2, we prove that the vertex-minimal flag triangulations of $\R P^2$ and $\Sp^1\times \Sp^1$ have 11 and 12 vertices, respectively. In general, we show that $8+3k$ (resp. $8+4k$) vertices suffice to obtain a flag triangulation of the connected sum of $k$ copies of $\R P^2$ (resp. $\Sp^1\times \Sp^1$). In dimension 3, we describe an algorithm based on the Lutz-Nevo theorem which provides supporting computational
	evidence for the following generalization of the Charney-Davis conjecture: for any flag 3-manifold,  $\gamma_2:=f_1-5f_0+16\geq 16 \beta_1$, where $f_i$ is the number of $i$-dimensional faces and $\beta_1$ is the first Betti number over a field $\field$. The conjecture is tight in the sense that for any value of $\beta_1$, there exists a flag 3-manifold for which the equality holds.
\end{abstract}
\begin{keywords}
{flag complexes, triangulations of manifolds, Charney-Davis conjecture}
\end{keywords}

\section{Introduction}\label{Section 1}
A simplicial complex is called flag, if all of its minimal non-faces have cardinality 2. The notion of flagness arises naturally from differential geometry. Gromov \cite{Gromov} noticed that when a piecewise Euclidean cubical complex is associated with the right-angled metric, the property that the complex is non-positively curved is equivalent to the condition that every vertex link in the complex is flag. Many classes of simplicial complexes with interesting combinatorial structures are flag; for example, the barycentric subdivisions of simplicial complexes, order complexes of posets, and Coxeter complexes. 

In this paper, we focus on a fundamental problem in combinatorial topology of flag complexes: what is the minimum number of vertices that a flag triangulation of a manifold can have? For instance, the boundary complex of a $d$-simplex gives the minimal triangulation of $\Sp^{d-1}$. In contrast, a flag triangulation of $\Sp^{d-1}$ requires at least $2d$ vertices. The vertex-minimal flag triangulation of $\Sp^{d-1}$ is given by the octahedral $(d-1)$-sphere and is unique. However, to the best knowledge of the authors, the vertex-minimal flag triangulations of any other manifolds (even the non-spherical surfaces) remain unknown.

We remark that it is not possible for us to detect flag triangulations of surfaces from existing references: most enumerative results search triangulated manifolds with up to a certain number of vertices, while we expect that the minimal flag triangulations require significantly more vertices than non-flag ones. For example, in \cite{Sulanke-Lutz} it is shown that there are 645592 distinct triangulations of $\Sp^1\times \Sp^1$ with up to 12 vertices, and we find that exactly one of these triangulations is flag. Even though there are enumerations of small triangulations of surfaces (see eg. \cite{Sulanke-Lutz,manifold_page}), in most of the references, only the $f$-vectors and types of manifolds are given, which is not sufficient to check flagness.

The first part of our paper aims at finding the minimal flag triangulations for several surfaces. The method is based on a theorem in \cite{Nevo-Lutz}, which states that any two flag homeomorphic PL manifolds can be connected via a sequence of edge subdivisions or admissible edge contractions. By computer search we find
\begin{itemize}
	\item two non-isomorphic 11-vertex flag triangulations of the real projective plane,
	\item one 12-vertex flag triangulation of the torus, and
	\item 28 non-isomorphic 14-vertex flag triangulations of the Klein bottle.
\end{itemize}

Based on the properties of flagness, we further prove that the 11-vertex and the 12-vertex flag triangulations of $\R P^2$ and $\Sp^1\times \Sp^1$ are indeed vertex-minimal. By defining a new connected sum for flag complexes, we also generate small flag triangulations of all other surfaces. Specifically, we show that $8+4k$ (resp. $8+3k$) vertices suffice to give a flag triangulation of the connected sum of $k$ tori (resp. real projective planes).

The second part of this paper explores $\gamma_2$-minimal flag 3-manifolds. The celebrated Charney-Davis conjecture \cite{Charney-Davis} asserts that for any flag $(2n-1)$-dimensional simplicial sphere $\Delta$, \[(-1)^n\left( 1-\frac{1}{2}f_0(\Delta)+\frac{1}{4}f_1(\Delta)-\dots +\left(-\frac{1}{2}\right)^{2n}f_{2n-1}(\Delta)\right) \geq 0,\] 
where $f_i$ is the number of $i$-dimensional faces. The Charney-Davis conjecture is an implication of the long-standing Hopf-Chern-Thurston conjecture from a combinatorial perspective; see \cite{Charney-Davis}, \cite{Forman} for motivation of the conjecture. The conjecture is proved for $n=2$ in \cite{DavisOkun-dimension 3 CD conjecture} using heavy machinery in differential geometry, and is now known as the Davis-Okun theorem. In this case, the above inequality can be rephrased using $\gamma$-numbers as $\gamma_2(\Delta):=f_1(\Delta)-5f_0(\Delta)+16\geq 0$. (We refer to \cite{Gal} for background on the $\gamma$-numbers and other lower-bound type conjectures for flag spheres.) It is natural to ask if the Davis-Okun theorem has an extension for flag 3-manifolds. In this paper we explore the connection between the minimal $\gamma_2$ for flag triangulations of a 3-manifold $M$ and the first Betti number of $M$. Based on all flag 3-manifolds that we've constructed, we propose a conjecture that $\gamma_2\geq 16 \beta_1$. Furthermore, we prove that for any integer $b\geq 0$, there exists a flag 3-manifold that attains $\gamma_2=16b$. See Section \ref{Section 5} for more discussion.

The structure of the paper is as follows. In Section \ref{Section 2}, we review the basic definitions and properties of flag complexes. In Section \ref{Section 3}, we present the algorithm for computer search and describe the small triangulations of several surfaces that we have found. In Section \ref{Section 4}, we prove that the minimal flag triangulations of $\R P^2$ and $\Sp^1\times \Sp^1$ have exactly 11 and 12 vertices, respectively, and describe how to construct small flag triangulations of other surfaces. We close in Section \ref{Section 5} by discussing potential extension to a manifold Charney-Davis conjecture based on computer search results on flag 3-manifolds. 

\section{Preliminaries}\label{Section 2}
A \emph{simplicial complex} $\Delta$ on a vertex set $V=V(\Delta)$ is a collection of subsets $\sigma\subseteq V$, called \emph{faces}, that is closed under inclusion. A subset $W$ of $V(\Delta)$ is a \emph{minimal non-face} of $\Delta$, or a \emph{missing face} of $\Delta$, if all proper subsets of $W$ are faces of $\Delta$ and $W\notin \Delta$. Two examples of simplicial complexes are the $d$-dimensional simplex on $V$, $\overline{V}:=\{\tau \ : \ \tau\subseteq V\}$, and its boundary complex, $\partial\overline{V}:=\{\tau : \tau\subsetneq V\}$. For $\sigma\in \Delta$, let $\dim\sigma:=|\sigma|-1$ and define the \emph{dimension} of $\Delta$, $\dim \Delta$, as the maximal dimension of its faces. A maximal under inclusion face of $\Delta$ is called a \emph{facet}. If all facets of $\Delta$ are of the same dimension, then $\Delta$ is called \emph{pure}. Let $E(\Delta)$ be the set of 1-faces in $\Delta$. For brevity, we write $\{v\}$ as $v$.

For a $(d-1)$-dimensional simplicial complex $\Delta$, we let $f_i = f_i(\Delta)$ be the number of $i$-dimensional faces of $\Delta$ for $-1\leq i\leq d-1$. The vector $(f_{-1}, f_0, \ldots, f_{d-1})$ is called the $f$\emph{-vector} of $\Delta$. If $\Delta$ is a simplicial complex and $\sigma$ is a face of $\Delta$, the \emph{link} of $\sigma$ in $\Delta$ is $\lk(\sigma,\Delta):=\{\tau-\sigma\in \Delta: \sigma\subseteq \tau\in \Delta\}$, and the \emph{star} of $\sigma$ in $\Delta$ is $\st(\sigma,\Delta):=\{\tau\in \Delta: \sigma\cup \tau\in \Delta\}$. When the context is clear, we will abbreviate the notation and write them as $\lk(\sigma)$ and $\st(\sigma)$ respectively. The \emph{restriction} of $\Delta$ to a vertex set $W$ is defined as $\Delta[W]:=\{\sigma\in \Delta:\sigma\subseteq W\}$. 

If $\Delta$ and $\Gamma$ are simplicial complexes defined on disjoint vertex sets, we define the \emph{join} of $\Delta$ and $\Gamma$ to be the simplicial complex $\Delta * \Gamma = \{\sigma \cup \tau : \sigma \in \Delta, \tau \in \Gamma\}$. Given a face $\sigma \in \Delta$, the \emph{stellar subdivision of $\Delta$ along $\sigma$} is $$\mathrm{sd}(\sigma, \Delta) = \{\tau \in \Delta : \tau \cap \sigma = \emptyset \} \cup (\bar{v} * \partial\bar{\sigma} * \lk(\sigma, \Delta)),$$ where $\bar{v}$ is a new vertex. If $e=\{u,v\}\in \Delta$, then we define the \emph{edge contraction of $\Delta$ along $e$} to be
	$$\mathrm{contr}(e, \Delta) = \{\tau \in \Delta : u \notin \tau\} \cup \{(\tau \cup v)\setminus u : u \subset \tau \in \Delta\}.$$

A $(d-1)$-dimensional simplicial complex $\Delta$ is called a \emph{simplicial $(d-1)$-manifold} (or \emph{triangulated $(d-1)$-manifold}) if its geometric realization $\|\Delta\|$ is homeomorphic to a manifold. A \emph{combinatorial $(d-1)$-manifold} (or \emph{PL $(d-1)$-manifold}) is a simplicial complex such that every vertex link is PL homeomorphic to the boundary of a $(d-1)$-simplex or the $(d-2)$-simplex. In particular, if $\Delta$ is a combinatorial manifold, then the \emph{boundary complex} of $\Delta$, denoted as $\partial \Delta$, consists of the empty face and the faces whose links are combinatorial balls. The boundary complex of a combinatorial $d$-ball is a combinatorial $(d-1)$-sphere. The faces that are not in the boundary complex are called the \emph{interior faces}. In dimension $d-1=2$ or $3$, the class of combinatorial manifolds and the class of simplicial manifolds are the same. However, there exist non-PL triangulations of the $(d-1)$-sphere on $d + 12$ vertices
for all $d-1\geq 5$, see \cite[Theorem 7]{Bjorner-Lutz}.

One advantage of working in the class of combinatorial manifolds can be explained by the following elegant theorem \cite{Alexander}.

\begin{theorem}[Alexander]\label{thm: Alexander}
	Two closed combinatorial manifolds are piecewise linear homeomorphic if and only if there exists a finite sequence of edge subdivisions and their inverses leading from one combinatorial manifold to the other. 
\end{theorem}
It is well-known that if an edge $e=\{a,b\}$ in a closed combinatorial manifold $\Delta$ satisfies the {\em link condition} $\lk(e, \Delta)=\lk(a, \Delta)\cap \lk(b, \Delta)$, then the edge contraction $\mathrm{contr}(e, \Delta)$ preserves the PL type of $\Delta$, see \cite{Nevo-obstruction}. 

There are two $(d-1)$-dimensional sphere bundles over the circle. We denote by
$\Sp^{d-2}\times \Sp^1$ the orientable sphere bundle and by $\Sp^{d-2}\twprod
\Sp^1$ the non-orientable sphere bundle. If $M$ is a $(d-1)$-dimensional
manifold, then we write $\#_i M$ as the (topological) connected sum of $i$
copies of $M$. The connected sum of simplicial complexes is defined as an analog
of the topological connected sum: let $\Delta$ and $\Gamma$ be pure
$(d-1)$-dimensional simplicial complexes and let $\sigma\in \Delta$, $\tau\in
\Gamma$ be facets. The connected sum of $\Delta$ and $\Gamma$,
$\Delta\#_{\sigma\sim \tau} \Gamma$ (or simply $\Delta\#\Gamma$, if $\sigma$ and
$\tau$ are understood in the context), is the complex obtained from $\Delta\cup
\Gamma$ by removing $\sigma,\tau$ and then identifying $\partial\overline{\sigma}$
with $\partial\overline{\tau}$.

Let $\chi(\Delta)=\sum_{i=0}^{d-1}(-1)^i\beta_i(\Delta)$ be the \emph{Euler characteristic} of the $(d-1)$-dimensional complex $\Delta$, where $\beta_i(\Delta)$ is the rank of the $i$th homology group of $\Delta$ computed with coefficients in $\Z$. By the Euler characteristic formula, $\chi(\Delta)=\sum_{i=0}^{d-1} (-1)^i f_i(\Delta)$. Hence the $f$-vector of a triangulated surface $\Delta$ can be expressed as $$f(\Delta)=(1, f_0(\Delta), 3(f_0(\Delta)-\chi(\Delta)),2(f_0(\Delta)-\chi(\Delta))).$$The following classification theorem of surfaces will come in handy in our discussion of flag triangulated surfaces, see \cite{Seifert-Threlfall}.
\begin{theorem}[Classification theorem of surfaces]\label{thm: classification}
	Every closed and connected surface is homeomorphic to one
	of the following:
	(1) a sphere,
	(2) a connected sum of tori,
	(3) a connected sum of projective planes.
\end{theorem}

A graph is an ordered pair $G=(V, E)$, such that $V$ is a finite set, and $E\subseteq \binom{V}{2}$. The graph of a simplicial complex $\Delta$ is $G(\Delta)=(V(\Delta), E(\Delta))$. A simplicial complex $\Delta$ is \emph{flag} if all minimal non-faces of $\Delta$ have cardinality 2; equivalently, $\Delta$ is the clique complex of $G(\Delta)$. For example, let $C_d^*=\conv\{\pm e_1, \dots, \pm e_d\}$ be the $d$-cross-polytope, where the $e_i$'s form the standard basis of $\R^d$. The boundary complex $\partial C_d^*$ is a flag $(d-1)$-sphere. The properties of flag complexes are described in the following lemma.
\begin{lemma}[{\cite[Lemma 5.2]{Nevo-Petersen}}]\label{lm: flag prop}
	Let $\Delta$ be a flag complex on vertex set $V$.
	\begin{enumerate}
		\item If $W\subseteq V(\Delta)$, then $\Delta[W]$ is also flag. 
		\item If $\sigma$ is a face in $\Delta$, then $\lk(\sigma)=\Delta[V(\lk(\sigma))]$. In particular, all links in a flag complex are also flag.
		\item If $W\subset V(\Delta)$, then $\|\Delta\|- \|\Delta[W]\|$  deformation retracts onto $\|\Delta[V - W]\|$.
		\item Any edge $\{v,v'\}$ in $\Delta$ satisfies the link condition $\lk(v)\cap \lk(v')=\lk(\{v,v'\})$. 
	\end{enumerate}
\end{lemma}
	
	
	

The flag analog of Alexander's theorem is the following theorem, see \cite{Nevo-Lutz}.
\begin{theorem}[Lutz-Nevo]\label{thm: Nevo-Lutz}
	Two flag simplicial complexes are piecewise linearly homeomorphic if and only if they can be connected by a sequence of flag complexes, each obtained from the previous one by either an edge subdivision or edge contraction.
\end{theorem}

\section{Algorithms and results}\label{Section 3}
It is well-known that the minimal flag triangulation of $\Sp^2$ is the octahedral 2-sphere with 6 vertices. Our goal in this section is to find the minimal flag triangulations of three other surfaces: the torus, the Klein bottle and the real projective plane. In the next section we will discuss how to generate small flag triangulations of all types of surfaces based on triangulations of these three surfaces.

Our implementation is based on the Lutz-Nevo theorem. Since in dimension $d=2, 3$, the class of combinatorial $d$-manifolds is the same as the class of triangulated $d$-manifolds, Theorem \ref{thm: Nevo-Lutz} guarantees that the minimal flag triangulation of a surface $M$ can be obtained by applying edge subdivisions and admissible edge contractions on a given (possibly very large) flag triangulation of $M$. In an edge subdivision, a chosen edge $e$ of the simplicial complex is divided into two edges, with the facets containing $e$ replaced by four facets containing the two new edges. The resulting complex is always flag. In an edge contraction, we choose an edge $e=\{a,b\}$ and identify $b$ with $a$; in other words, the edge $e$ is contracted to the vertex $a$. Faces of the form $F\cup \{b\}$ are replaced by the faces $F\cup\{a\}$. However, edge contraction does not always preserve flagness. The resulting complex is flag if and only if the edge contracted is not contained in any induced 4-cycle of the original simplicial complex; see \cite[Corollary 6.2]{Nevo-Lutz}. We call such an edge an \emph{admissible} edge. 

Our computer search algorithm is as follows: we first build a relatively small flag triangulation of a given surface and apply a random sequence of edge subdivisions on the complex until the number of vertices reaches a set number. Then a sequence of admissible edge contractions is performed until a local minimum on $f_0$ is attained, i.e., no more admissible edges exist in the complex. The above process of edge subdivisions followed by edge contractions is iterated a given number of times. In each iteration, we keep track of the number of vertices in the complexes.

In what follows, we summarize the smallest flag triangulations of $\R P^2$, $\Sp^1\times \Sp^1$ and $\Sp^1\twprod \Sp^1$ found by our implementation. 

\begin{figure}[h]
	\caption{Two 11-vertex flag triangulations of $\R P^2$}
	\begin{multicols}{2}
		\centering
		\begin{tikzpicture}[scale=1.3]
\node (8) at (0,0) {8};
\node (9) at (0,1) {9};
\node (1) at (0.951,.309) {1};
\node (4) at (.588,-.809) {4};
\node (7) at (-.588,-.809) {7};
\node (11) at (-.951,.309) {11};
\node (3a) at (0,2) {3};
\node (2a) at (.951,1.309) {2};
\node (6a) at (1.902,.618) {6};
\node (5a) at (1.539,-.5) {5};
\node (10a) at (1.176,-1.618) {10};
\node (3b) at (0,-1.618) {3};
\node (2b) at (-1.176,-1.618) {2};
\node (6b) at (-1.539,-.5) {6};
\node (5b) at (-1.902,.618) {5};
\node (10b) at (-.951,1.309) {10};
\draw (9)--(8)--(11)--(5b)--(10b)--(3a)--(2a)--(6a)--(1)--(8);
\draw (10b)--(2a);
\draw (5b)--(6b)--(2b)--(3b)--(10a);
\draw (6a)--(5a)--(10a)--(4)--(8)--(7)--(2b);
\draw (4)--(3b)--(7)--(4)--(5a)--(1)--(4);
\draw (9)--(2a)--(1)--(9)--(10b)--(11)--(6b)--(7)--(11)--(9);
		\end{tikzpicture}
		
		\begin{tikzpicture}[scale=1.3]
\node (8) at (0,0) {8};
\node (9) at (0,1) {9};
\node (1) at (0.951,.309) {1};
\node (4) at (.588,-.809) {4};
\node (7) at (-.588,-.809) {7};
\node (11) at (-.951,.309) {11};
\node (3a) at (0,2) {3};
\node (2a) at (.951,1.309) {2};
\node (6a) at (1.902,.618) {6};
\node (5a) at (1.539,-.5) {5};
\node (10a) at (1.176,-1.618) {10};
\node (3b) at (0,-1.618) {3};
\node (2b) at (-1.176,-1.618) {2};
\node (6b) at (-1.539,-.5) {6};
\node (5b) at (-1.902,.618) {5};
\node (10b) at (-.951,1.309) {10};
\draw (3a)--(9)--(8)--(11)--(5b)--(10b)--(3a)--(2a)--(6a)--(1)--(8);
\draw (5b)--(6b)--(2b)--(3b)--(10a);
\draw (6a)--(5a)--(10a)--(4)--(8)--(7)--(2b);
\draw (4)--(3b)--(7)--(4)--(5a)--(1)--(4);
\draw (9)--(2a)--(1)--(9)--(10b)--(11)--(6b)--(7)--(11)--(9);
		\end{tikzpicture}
		\label{figure: RP^2}
	\end{multicols}
\end{figure}
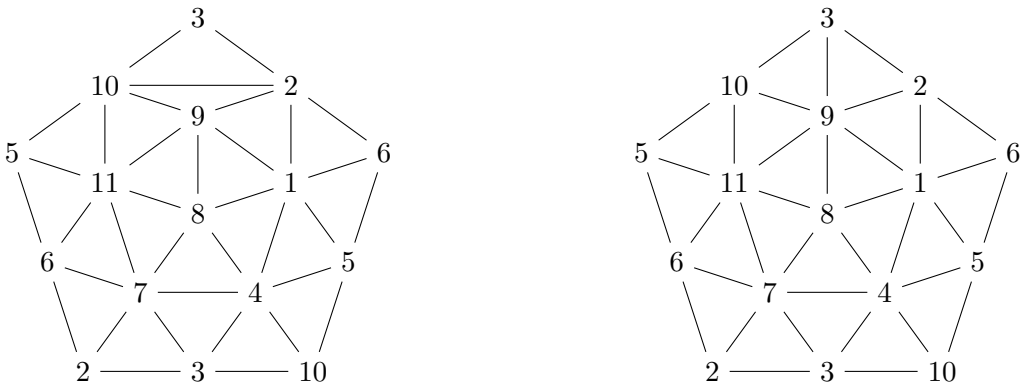
\begin{table}[h]
	\caption{Facets of two flag 11-vertex triangulations of $\R P^2$}\label{table: RP^2}
	\begin{multicols}{2}
		\centering
		\begin{tabular}{c c c c c c}\hline\hline
			5 6 11 &6 7 11 &7 8 11 &8 9 11 &9 10 11 \\ \hline
			5 10 11 & 1 2 9&2 9 10 & 2 3 10& 3 4 10\\ \hline
			4 5 10 & 1 4 5 & 1 5 6&1 2 6 &2 6 7 \\ \hline
			2 3 7 &3 4 7 &4 7 8 & 1 4 8 & 1 8 9 \\ \hline\hline
			
		\end{tabular}
		
		\begin{tabular}{c c c c c c}\hline\hline
			5 6 11 &6 7 11 &7 8 11 &8 9 11 &9 10 11 \\ \hline
			5 10 11 & 1 2 9&2 3 9& 3 9 10& 3 4 10\\ \hline
			4 5 10 & 1 4 5 & 1 5 6&1 2 6 &2 6 7 \\ \hline
			2 3 7 &3 4 7 &4 7 8 & 1 4 8 & 1 8 9 \\ \hline\hline
			
		\end{tabular}
	\end{multicols}
\end{table}
\begin{example}[Triangulation of $\mathbb RP^2$]
    We started with an 11-vertex flag triangulation of $\mathbb RP^2$ as shown in the left of Figure \ref{figure: RP^2}. Our program found another non-isomorphic construction of $\mathbb RP^2$ with 11 vertices (see the right of Figure \ref{figure: RP^2}).
Observe that, as the picture suggests, the right-hand triangulation has
automorphism group $\operatorname{Dih}_5$ (the symmetry group of the pentagon), while the left-hand triangulation does not admit an automorphism of order five. 
 The facets of these two triangulations are listed in Table \ref{table: RP^2}. In particular, they differ by one bistellar flip, i.e., by replacing the 2-faces $\{2,3,10\}, \{2,9,10\}$ with $\{2,3,9\}, \{3,9,10\}$. We will prove in the next section that the minimal flag triangulations of $\R P^2$ indeed have 11 vertices.
\end{example}

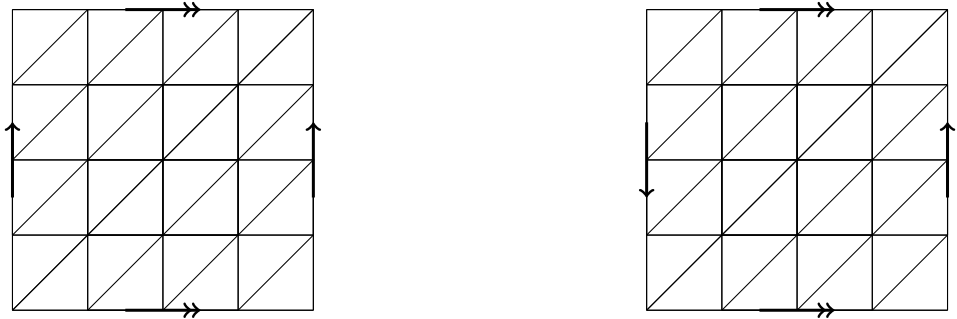
\begin{figure}[h]
	\caption{The 16-vertex flag triangulations of the torus (left) and the Klein bottle (right)}
	\begin{multicols}{2}
		\centering
		\begin{tikzpicture}[every edge/.style = {draw=black,very thick},vrtx/.style args = {#1/#2}{%
			circle, draw, thick, fill=white,
			minimum size=3mm, label=#1:#2}
		]
		\foreach \a in {-2,-1,0,1}
		{\draw (-2, \a)--(-\a,2);
			\draw (\a,-2)--(2,-\a);}
		
		\foreach \x in {-2,-1,0,1,2}
		\foreach \y in {-2,-1,0,1,2}
		{
			\draw (\x,\y)--(-\x,\y);
			\draw (\x, \y)--(\x,-\y);
		}
		\draw[->>, line width=0.4mm] (-0.5,-2) -- (0.5,-2);
		\draw[->>, line width=0.4mm] (-0.5,2) -- (0.5,2);
		\draw[->, line width=0.4mm] (-2,-0.5) -- (-2,0.5);
		\draw[->, line width=0.4mm] (2,-0.5) -- (2,0.5);
		\end{tikzpicture}
		\begin{tikzpicture}[every edge/.style = {draw=black,very thick},vrtx/.style args = {#1/#2}{%
			circle, draw, thick, fill=white,
			minimum size=3mm, label=#1:#2}
		]
		\foreach \a in {-2,-1,0,1}
		{\draw (-2, \a)--(-\a,2);
			\draw (\a,-2)--(2,-\a);}
		
		\foreach \x in {-2,-1,0,1,2}
		\foreach \y in {-2,-1,0,1,2}
		{
			\draw (\x,\y)--(-\x,\y);
			\draw (\x, \y)--(\x,-\y);
		}
		\draw[->>, line width=0.4mm] (-0.5,-2) -- (0.5,-2);
		\draw[->>, line width=0.4mm] (-0.5,2) -- (0.5,2);
		\draw[->, line width=0.4mm] (-2,0.5) -- (-2,-0.5);
		\draw[->, line width=0.4mm] (2,-0.5) -- (2,0.5);
		
		\end{tikzpicture}
	\end{multicols}	
    \label{figure: torus_candidate}
\end{figure}

\begin{example}[Triangulation of the torus]
    We started with a flag triangulation of the torus with 16 vertices as shown in the left of Figure \ref{figure: torus_candidate}. A unique flag triangulation with 12 vertices was found in the searching process. The facets are listed in Table \ref{table: torus}. We will show in the next section that this is indeed the unique minimal flag triangulation of the torus, see Figure \ref{Figure: S1S1}.
\end{example}
 \begin{table}[h]
 	\caption{Facets of the flag 12-vertex triangulation of $\Sp^1\times \Sp^1$}
 	\begin{center}
 		\begin{tabular}{c c c c c c}\hline\hline
 			1 2 3& 1 2 5 & 1 3 11 & 1 4 5 & 1 4 12 & 1 11 12\\ \hline 
 			2 3 8& 2 5 6& 2 6 7& 2 7 8& 3 8 9 & 3 9 10 \\ \hline 
 			3 10 11& 4 5 10&4 7 8 &4 7 10 &4 8 12 & 5 6 9\\ \hline 
 			5 9 10 &6 7 11 & 6 9 12 & 6 11 12& 7 10 11& 8 9 12\\ \hline\hline
 			
 		\end{tabular}
 	\end{center}
 \label{table: torus}
 \end{table}
     
\begin{example}[Triangulation of the Klein bottle]
    We started the searching process with a flag triangulation of the Klein bottle with 16 vertices as shown in Figure \ref{figure: torus_candidate}. Our program suggested that a minimal flag triangulation of the Klein bottle has 14 vertices, and there are at least 28 non-isomorphic such triangulations. We will see in the next section how to obtain several 14-vertex triangulations from 11-vertex flag triangulations of $\R P^2$.
\end{example}

\begin{remark}
	The minimal flag triangulation of $\Sp^1\times \Sp^1$ is exactly the vertex-transitive 2-manifold $^2 12^{83}_1$ found in \cite{manifold_page}. It admits the group action of $S_4\times S_3$. 
\end{remark}

\section{Proof of minimality}\label{Section 4}
In Section \ref{Section 3}, we saw that there exist flag triangulations of $\R P^2$ and $\Sp^1\times \Sp^1$ with 11 and 12 vertices, respectively. In this section, we prove that indeed our constructions give the minimal flag triangulations. The following lemma provides a necessary condition for a flag triangulation to be vertex-minimal.
\begin{lemma}\label{lm: every edge in a 4-cycle}
	Let $\Delta$ be a minimal flag triangulated surface. Then every edge is in an induced 4-cycle of $\Delta$.
\end{lemma}
\begin{proof}
	Suppose not, then there is an edge $e=\{u,v\}\in \Delta$, such that it is not in any induced 4-cycle of $\Delta$. Then $\{u,v\}$ can be admissibly contracted into $\Delta'$, where $\Delta'$ is
	homeomorphic to $\Delta$ and $\Delta'$ contains no induced 3-cycle, i.e., $\Delta$ is flag. This contradicts that $\Delta$ is the minimal flag triangulation. 
\end{proof}

\subsection{The minimal flag triangulation of $\mathbb{R}P^2$}
In what follows, we denote by $(w_1, w_2, \dots, w_n)$ the $n$-cycle with edges $\{w_i, w_{i+1}\}$ for $1\leq i\leq n-1$ and $\{w_n, w_1\}$.
\begin{lemma}\label{lm: no adjacent 4-cycle links}
	Let $\Delta$ be a minimal flag triangulation of a surface $M$. If there is an edge $e=\{v_1, v_2\}\in \Delta$ such that both of the links $\lk(v_1), \lk(v_2)$ are 4-cycles, then $M\cong \Sp^2$ and $\Delta$ is the octahedral sphere.
\end{lemma}
\begin{proof} Assume that $\lk(v_1)=(v_5, v_3, v_2, v_4)$ and $\lk(v_2)=(v_1, v_3, v_6, v_4)$, as shown in the left of Figure \ref{figure: no 4,5-cycles}. By Lemma \ref{lm: flag prop}, $\lk(v_1)\cap \lk(v_2)=\lk(\{v_1,v_2\})$ is the union of two vertices $v_3, v_4$ and hence $v_5, v_6$ are distinct. Since $\Delta$ is a minimal flag triangulation, the edge $\{v_5,v_1\}$ must be in at least one induced 4-cycle $C$ of $\Delta$. Since $C$ is induced, $\{v_3,v_1\}$ or $\{v_4,v_1\}$ cannot be in $C$. So the edge $\{v_1,v_2\}$ is in $C$. Similarly, we have $\{v_2,v_6\}\subset C$. Hence $C=(v_5,v_1, v_2,v_6)$. Note that $G(\st(v_1)\cup \st(v_2)\cup \{v_5,v_6\})$ is the graph of an octahedral sphere. By flagness, this octahedral sphere is a subcomplex of $\Delta$, and hence $\Delta$ must be the octahedral sphere.
\end{proof}

\begin{figure}[h]
	\centering	
	\begin{multicols}{2}
		\begin{tikzpicture}
		\node at (-1,0)[label=above:$v_1$]{};
		\node at (1,0)[label=above:$v_2$]{};
		\node at (0,2)[label=right:$v_3$]{};
		\node at (0,-2)[label=right:$v_4$]{};
		\node at (-3,0)[label=above:$v_5$]{};
		\node at (3,0)[label=above:$v_6$]{};
		\draw[thick](-3,0)--(3,0);
		\draw[thick](-1,0)--(0,-2);
		\draw[thick](-1,0)--(0,2);
		\draw[thick](1,0)--(0,-2);
		\draw[thick](1,0)--(0,2);
		\draw[thick](-3,0)--(0,-2);
		\draw[thick](3,0)--(0,2);
		\draw[thick](-3,0)--(0,2);
		\draw[thick](3,0)--(0,-2);
		
		\end{tikzpicture}
		
		\begin{tikzpicture}
		\node at (-1,0)[label=above:$v_1$]{};
		\node at (1,0)[label=above:$v_2$]{};
		\node at (0,2)[label=left:$v_3$]{};
		\node at (0,-2)[label=left:$v_4$]{};
		\node at (-3,0)[label=left:$v_5$]{};
		\node at (2,2)[label=right:$v_6$]{};
		\node at (2,-2)[label=right:$v_7$]{};
		
		\draw[thick](-3,0)--(1,0);
		\draw[thick](-1,0)--(0,-2);
		\draw[thick](-1,0)--(0,2);
		\draw[thick](1,0)--(0,-2);
		\draw[thick](1,0)--(0,2);
		\draw[thick](-3,0)--(0,-2);
		\draw[thick](2,2)--(0,2);
		\draw[thick](-3,0)--(0,2);
		\draw[thick](2,-2)--(0,-2);
		\draw[thick](1,0)--(2,2);
		\draw[thick](2,-2)--(1,0);
		\draw[thick](2,-2)--(2,2);
		\end{tikzpicture}
		
	\end{multicols}
	
	\caption{Links in the proof of Lemma \ref{lm: no adjacent 4-cycle links} (left), where $\deg v_1= \deg v_2=4$, and Lemma \ref{lm: no adjacent 4-cycle, 5-cycles} (right), where $\deg v_1=4, \deg v_2=5$.}
	\label{figure: no 4,5-cycles}
\end{figure}
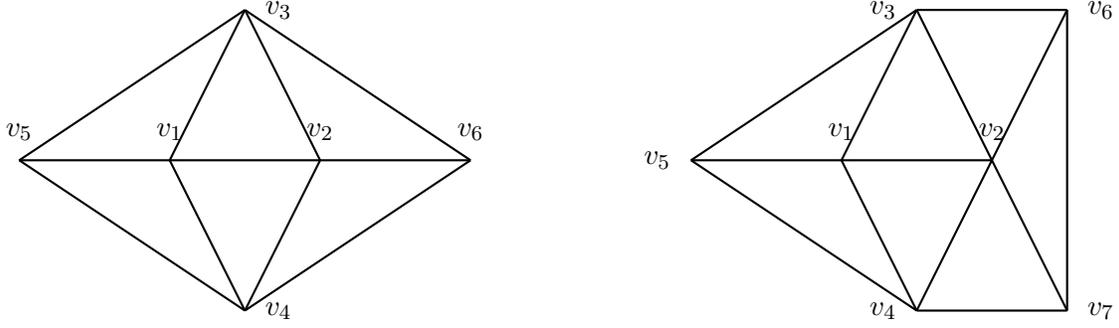
We can strengthen Lemma \ref{lm: no adjacent 4-cycle links} as follows:
\begin{lemma}\label{lm: no adjacent 4-cycle, 5-cycles}
	Let $\Delta$ be a minimal flag triangulation of a surface $M$. Then no adjacent vertices can have degree 4 and 5, respectively, in $\Delta$.
\end{lemma}

\begin{proof} Assume that $\lk(v_1)=(v_5, v_3, v_2, v_4)$ and $\lk(v_2)=(v_1, v_3, v_6, v_7,v_4)$, as shown in the right of Figure \ref{figure: no 4,5-cycles}. Using the same proof as in Lemma \ref{lm: no adjacent 4-cycle links}, $v_5, v_6, v_7$ are distinct vertices and the induced 4-cycle $C$ that contains the edge $\{v_5,v_1\}$ must also contain $\{v_1,v_2\}$. By symmetry and without loss of generality, we may assume that $\{v_2,v_6\}$ is in $C$. Hence $C=(v_5, v_1, v_2, v_6)$. By flagness $\lk(v_3)$ must be the 4-cycle $C$. Now both $\lk(v_3)$ and  $\lk(v_1)$ are 4-cycles and $\{v_1,v_3\}\in \Delta$. By Lemma \ref{lm: no adjacent 4-cycle links}, $\Delta$ is the octahedral sphere. This contradicts that $\lk(v_2)$ is a 5-cycle. 
\end{proof}

Now we are ready to prove our main theorem.
\begin{theorem}
	Let $\Delta$ be a minimal flag triangulation of a surface $M\neq \Sp^2$. Then $f_0(\Delta)\geq 11$.
\end{theorem}
\begin{proof}
	Let $F=\{v_1, v_2, v_3\}\in \Delta$. By Lemmas \ref{lm: no adjacent 4-cycle links} and \ref{lm: no adjacent 4-cycle, 5-cycles}, either $(\deg v_1, \deg v_2, \deg v_3)=(4, \geq 6, \geq 6)$ or $(\geq 5, \geq 5, \geq 5)$. 
	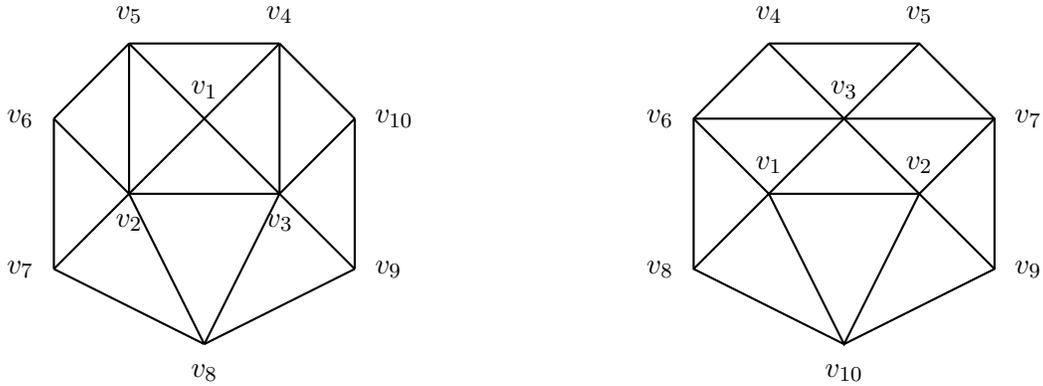
\begin{figure*}
		\begin{multicols}{2}
			\centering
			\begin{tikzpicture}
			\node at (-1,0)[label=below:$v_2$]{};
			\node at (1,0)[label=below:$v_3$]{};
			\node at (0,1)[label=above:$v_1$]{};
			\node at (0,-2)[label=below:$v_8$]{};
			\node at (2,-1)[label=right:$v_9$]{};
			\node at (-2,-1)[label=left:$v_7$]{};
			\node at (-2,1)[label=left:$v_6$]{};
			\node at (-1,2)[label=above:$v_5$]{};
			\node at (1,2)[label=above:$v_4$]{};
			\node at (2,1)[label=right:$v_{10}$]{};
			\draw[thick](-1,0)--(0,1);
			\draw[thick](-1,0)--(1,0);
			\draw[thick](-1,0)--(-2,1);
			\draw[thick](-1,0)--(-2,-1);
			\draw[thick](-1,0)--(0,-2);
			\draw[thick](-1,0)--(-1,2);
			\draw[thick] (0,1)-- (-1,2);
			\draw[thick] (0,1)--(1,2);
			\draw[thick] (1,2)--(2,1);
			
			\draw[thick] (1,0)--(0,1) ;
			\draw[thick] (1,0)--(2,1) ;
			\draw[thick] (1,0)--(2,-1) ;
			\draw[thick] (0,-2)--(2,-1) ;
			\draw[thick] (0,-2)--(1,0) ;
			
			\draw[thick] (2,1)--(2,-1) ;
			
			\draw[thick] (-2,-1)--(0,-2) ;
			\draw[thick] (-2,-1)--(-2,1);
			
			\draw[thick] (-1,2)--(-2,1);
			\draw[thick] (1,2)--(1,0);
			\draw[thick] (1,2)--(-1,2);
			\end{tikzpicture}
			\begin{tikzpicture}
			\node at (-1,-1)[label=above:$v_1$]{};
			\node at (0,0)[label=above:$v_3$]{};
			\node at (1,-1)[label=above:$v_2$]{};
			\node at (-1,1)[label=above:$v_4$]{};
			\node at (1,1)[label=above:$v_5$]{};
			\node at (-2,0)[label=left:$v_6$]{};
			\node at (2,0)[label=right:$v_7$]{};	\node at (-2,-2)[label=left:$v_8$]{};	\node at (2,-2)[label=right:$v_9$]{};	\node at (0,-3)[label=below:$v_{10}$]{};
			\draw[thick](-1,1)--(0,0)--(1,-1)--(2,-2);
			\draw[thick](1,1)--(-2,-2);	\draw[thick](-1,1)--(1,1)--(2,0)--(2,-2)--(0,-3)--(-2,-2)--(-2,0)--(-1,1);	\draw[thick](-2,0)--(2,0);	\draw[thick](-2,0)--(-1,-1)--(0,-3)--(1,-1)--(2,0);	\draw[thick](-1,-1)--(1,-1);
			\end{tikzpicture}
		
		\end{multicols}
	\caption{Links for case 1 of the proof of Theorem \ref{thm: minimal RP2}, where $(\deg v_1, \deg v_2, \deg v_3)=(4,6,6)$, and case 2 (right), where $(\deg v_1, \deg v_2, \deg v_3)=(5,5,6)$.}\label{fig: case 1,2}
	\end{figure*}
	
	{\bf Case 1:} The vertex $v_1$ is of degree 4. Since $\lk({v_i})\cap\lk({v_j})=\lk(\{v_i, v_j\})$ consists of two distinct vertices for any distinct $i,j\in\{1,2,3\}$, by the inclusion-exclusion principle and flagness of $\Delta$ we have that
	\[f_0(\Delta)\geq f_0(\cup_{i=1}^3 \lk(v_i))= \sum_{i=1}^3f_0(\lk(v_i))- \sum_{1\leq i<j\leq 3}f_0(\lk(v_i)\cap \lk(v_j)) = \sum_{i=1}^3f_0(\lk(v_i))-6\geq 10.\]
	In particular, if $f_0(\Delta)=10$, then it follows that $(\deg v_1, \deg v_2, \deg v_3)=(4, 6, 6)$. Assume that the links of $v_1,v_2, v_3$ are as shown in the left of Figure \ref{fig: case 1,2}. Since $\deg v_1=4$, by Lemma \ref{lm: no adjacent 4-cycle, 5-cycles} $\lk(v_4)$ is not a 4-cycle and hence $\{v_5,v_{10}\}\notin \Delta$. Since $\lk(v_2)$ is flag, $\{v_5, v_3\}$, $\{v_5,v_7\}$ and $\{v_5,v_8\}$ are not edges of $\Delta$. Hence $v_5$ can have at most degree 5 in $\Delta$. However by Lemma \ref{lm: no adjacent 4-cycle links}, $\deg v_5 \geq 6$, a contradiction. Therefore, $f_0(\Delta)\geq 11$.
	
	{\bf Case 2:} Every vertex in $\Delta$ is of degree at least 5. It follows from the Euler Characteristic formula that $f(\Delta)= (1, f_0(\Delta), 3(f_0(\Delta)-\chi(\Delta)), 2(f_0(\Delta)-\chi(\Delta)))$; here $\chi(\Delta) \leq 1$ as $M\neq \Sp^2$. If all vertices of $\Delta$ are of degree 5, by double counting we have
	\[6(f_0(\Delta)-\chi(\Delta))=2f_1(\Delta)=\sum_{v\in \Delta} f_0(\lk(v))=5f_0(\Delta).\]That is, $f_0(\Delta)=6\chi(\Delta)\leq 6$, a contradiction.
	
	Without loss of generality, assume  $(\deg v_1, \deg v_2, \deg v_3)=(\geq5,\geq 5,\geq 6)$. As before, we have $f_0(\Delta)\geq f_0(\cup_{i=1}^3 \lk(v_i))\geq 5+5+6-6= 10$. Suppose $f_0(\Delta)=10$, in which case $(\deg v_1, \deg v_2, \deg v_3)=(5,5,6)$ and the links of $v_i$ are as shown in the right of Figure \ref{fig: case 1,2}. 
	By Lemma \ref{lm: no adjacent 4-cycle, 5-cycles}, $\lk(v_{10})$ and $\lk(v_6)$ are not 4-cycles and hence $\{v_8,v_9\}, \{v_8,v_4\}\notin \Delta$. Since $\deg(v_8)\geq 5$, we must have $\{v_8,v_5\},\{v_8,v_7\}\in \Delta$ and $\lk(v_8)$ is equal to $(v_6, v_1, v_{10}, v_5, v_7)$ or $(v_6, v_1, v_{10}, v_7, v_5)$. However, both $\lk(v_2)$ and $\lk(v_3)$ are flag and $\{v_7, v_{10}\}$, $\{v_6, v_7\}$ cannot be edges of $\Delta$. This leads to a contradiction.
\end{proof}

In Section \ref{Section 3} we showed the existence of an 11-vertex flag triangulation of $\R P^2$. We immediately obtain the following:
\begin{theorem}\label{thm: minimal RP2}
	A minimal flag triangulation of $\R P^2$ has 11 vertices.
\end{theorem}
\subsection{The minimal flag triangulation of $\Sp^1\times \Sp^1$}
In what follows, we show that the minimal flag triangulation of $\Sp^1\times \Sp^1$ has 12 vertices and furthermore, it is unique.  
\begin{lemma}\label{lm: no 7-cycle link}
	Let $\Delta$ be a minimal flag triangulation of $\Sp^1\times \Sp^1$. Then $f_0(\Delta)\leq 12$ and every vertex of $\Delta$ is of degree $\leq 6$.
\end{lemma}
\begin{proof}
	We know that $f_0(\Delta)\leq 12$, since in Section \ref{Section 3} we found a flag triangulation of $\Sp^1\times \Sp^1$ with 12 vertices. Assume that there is a vertex $v$ of degree $\geq 7$. Let $W=V(\st(v))$ and $W^c=V(\Delta)-W$. By Lemma \ref{lm: flag prop}, $\|\Delta-\Delta[W]\|=\|\Delta-\st(v)\|$ deformation retracts onto $\|\Delta[W^c]\|$. On the other hand, since $|W^c|=|V(\Delta)|-|W|\leq 12-8\leq 4$ and $\Delta[W^c]$ is flag, it follows that either $\Delta[W^c]$ is the 4-cycle or it is contractible. Hence $$2=\beta_1(\|\Delta-v\|)=\beta_1(\|\Delta-\st(v)\|)=\beta_1(\|\Delta[W^c]\|)\leq 1,$$ a contradiction.
\end{proof}
\begin{lemma}
	A minimal flag triangulation $\Delta$ of $\Sp^1\times \Sp^1$ has 12 vertices and each vertex is of degree 6. 
\end{lemma}
\begin{proof}
	Since $\chi(\Delta)=0$, by the Euler characteristic, $f(\Delta)=(1, f_0, 3f_0, 2f_0)$. By Lemma \ref{lm: no 7-cycle link},
	\[6f_0(\Delta)=2f_1(\Delta)=\sum_{v\in\Delta}f_0(\lk(v))\leq 6f_0(\Delta).\]
	Hence $f_0(\lk(v))=6$ for every vertex of $\Delta$. If $\{v_1, v_2, v_3\}$ is a 2-face of $\Delta$, then \[12\geq f_0(\Delta)\geq f_0(\cup_{i=1}^3 \lk(v_i))\geq \sum_{i=1}^3f_0(\lk(v_i))-\sum_{1\leq i<j\leq 3}f_0(\lk(\{v_i,v_j\})=3\cdot 6-3\cdot 2=12.\] This proves the claim.
\end{proof}

\begin{theorem}
	The minimal flag triangulation of $\Sp^1\times \Sp^1$ has 12 vertices and is unique.
\end{theorem}
\begin{proof}
	Let $\{v_1, v_2, v_3\}$ be a 2-face of $\Delta$. By the above lemma, $|V(\cup_{i=1}^3 \st(v_i))|=|V(\Delta)|=12$. Furthermore, $G(\Delta)$ contains the subgraph $G(\cup_{i=1}^3 \st(v_i))$, as shown in the blue part of Figure \ref{Figure: S1S1}. 
		
	Since $\lk(v_i)=\Delta[V(\lk(v_i))]$, by flagness $v_5$ is not connected to $v_7, v_8, v_3, v_{11}, v_{12}$. On the other hand, $\deg(v_5)=6$. Hence $v_5$ must be connected to $v_9, v_{10}$. Applying the same argument to $v_8, v_{11}$, we have $\{v_8, v_4\}, \{v_8, v_{12}\}, \{v_{11}, v_6\}, \{v_{11}, v_7\}\in \Delta$.
	
	It is left to decide the remaining 6 edges. Since $v_4$ is of degree 6, $v_4$ is connected to two vertices among $v_7, v_9, v_{10}$. However, if both $\{v_4, v_9\}$ and $\{v_4, v_{10}\}$ are edges of $\Delta$, $\{v_4, v_5, v_9, v_{10}\}$ will form a clique in $G(\Delta)$ and by flagness of $\Delta$ form a 3-face of $\Delta$, which is not possible. Hence $\{v_4, v_7\}\in \Delta$. Similarly, we apply the same argument on the vertices $v_7, v_{10}, v_6, v_9, v_{12}$ respectively. This shows that $\{v_7, v_{10}\}$,  $\{v_{10}, v_4\}$, $\{v_6, v_9\}$, $\{v_9, v_{12}\}$ and $\{v_{12}, v_6\}$ are edges of $\Delta$. This gives all 36 edges in $\Delta$ and the graph is isomorphic to the 12-vertex construction in Section \ref{Section 3}.
\end{proof}
\begin{figure}[h]
	\centering
	\includegraphics[scale=0.8]{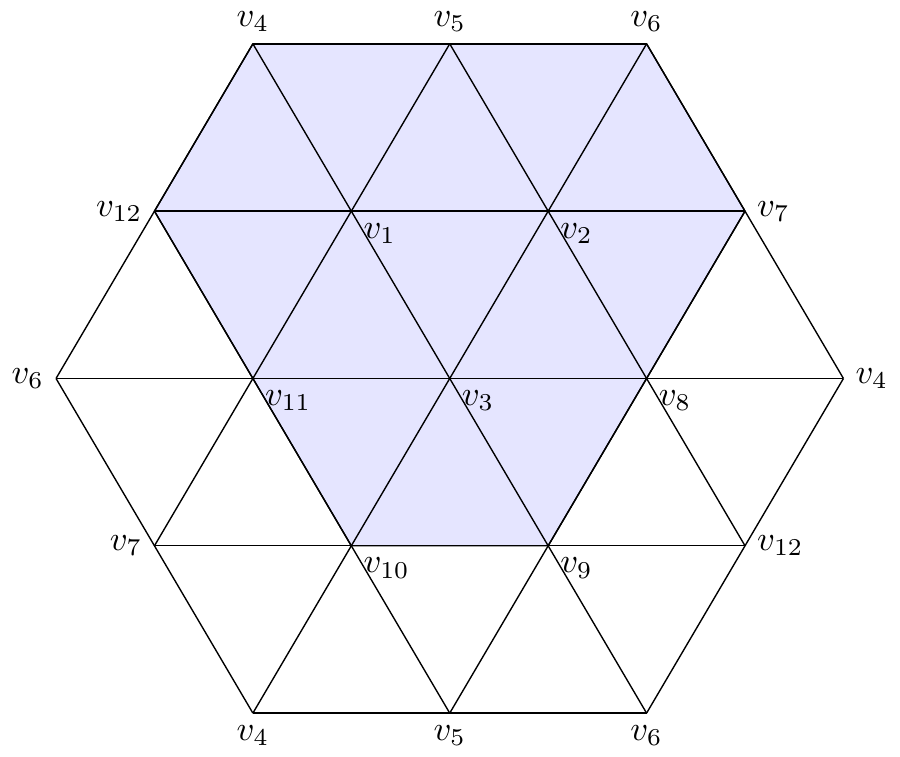}
	\caption{The triangulation of $\Sp^1\times \Sp^1$ obtained by identifying three pairs of (oriented) opposite sides in the triangulated hexagon.}
	\label{Figure: S1S1}
\end{figure}

\subsection{Generating small flag triangulations of all surfaces}
By the classification theorem of surfaces (see Theorem \ref{thm: classification}), a surface is either a 2-sphere, or the connected sum of $\Sp^1\times \Sp^1$, or the connected sum of $\R P^2$. Hence to construct a small flag triangulation of all surfaces, it suffices to show how to take the connected sum of surfaces efficiently while preserving flagness.

\begin{definition}\label{def: flag connected sum}
	Let $\Delta_1$ and $\Delta_2$ be the flag triangulations of $(d-1)$-manifolds $M_1$ and $M_2$, respectively. Let $W_1\subset V(\Delta_1)$ and $W_2\subset V(\Delta_2)$. Assume that $\Delta_1[W_1]$ and $\Delta_2[W_2]$ are simplicial $(d-1)$-balls and that  $\partial \Delta_1[W_1]$ and $\partial \Delta_2[W_2]$ are flag simplicial $(d-2)$-spheres; furthermore, assume that $\phi: \partial \Delta_1[W_1]\to \partial \Delta_2[W_2]$ is a simplicial isomorphism. Define a simplicial complex $\Delta=\Delta_1\#_\phi\Delta_2$ by 1) removing the interior faces of $\Delta_1[W_1]$ and $\Delta_2[W_2]$ and 2) gluing $\Delta_1$ and $\Delta_2$ by identifying $\partial (\Delta_1[W_1])$ with $\phi(\partial (\Delta_1[W_1]))=\partial (\Delta_2[W_2])$. We say $\Delta$ is a \emph{flag connected sum} of $\Delta_1$ and $\Delta_2$ under $\phi$.
\end{definition}
We remark that $\Delta_1[W_1]$ and $\Delta_2[W_2]$ are not necessarily
isomorphic in Definition \ref{def: flag connected sum}. The following lemma
shows that Definition \ref{def: flag connected sum} generates a new flag complex.
\begin{lemma}\label{lm: flag connected sum}
	If $\Delta_1$ and $\Delta_2$ are flag triangulations of $(d-1)$-manifolds $M_1$ and $M_2$, respectively, then the flag connected sum $\Delta=\Delta_1\#_\phi\Delta_2$ is a flag triangulation of $M_1\#M_2$. 
\end{lemma} 
\begin{proof}
	First $\Delta$ triangulates $M_1\#M_2$ because $\|\Delta_1\|=M_1$, $\|\Delta_2\|=M_2$ and $\Delta_1[W_1]$ and $\Delta_2[W_2]$ are simplicial full-dimensional balls. Assume that $\Delta$ is not flag and $F$ is a missing face of $\Delta$ of dimension $\geq 2$. If there are two vertices $v_1, v_2 \in F$ such that $v_i\in \Delta_i$ but $v_i\notin W_i$ for $i=1,2$, then by the construction, $v_1$ is not connected to $v_2$ in $\Delta$, contradicting the fact that $F$ is a missing face of dimension $\geq 2$. So all the vertices of $F$ are either in $\Delta_1$ or $\Delta_2$; without loss of generality assume that they are in $\Delta_1$. By flagness of $\Delta_1$, we have $F\in \Delta_1$. But since $F\notin \Delta$, $F$ must be an interior face of $\Delta_1[W_1]$ which is removed in the flag connected sum. In particular, $F$ is a missing face of $\partial \Delta_1[W_1]$, contradicting the flagness of $\partial \Delta_1[W_1]$.
\end{proof}

The above lemma together with the minimal flag triangulations of $\Sp^1\times \Sp^1$ and $\R P^2$ lead to the following upper bound on the number of vertices of minimal flag triangulated surfaces.
\begin{theorem}\label{thm: small flag triangulations of all surfaces}
	Let $k\geq 1$ be an integer. A minimal flag triangulation of the connected sum of $k$ tori requires at most $8+4k$ vertices, while a minimal flag triangulation of the connected sum of $k$ real projective planes requires at most $8+3k$ vertices.
\end{theorem}
\begin{proof}
The proof is by induction on $k$. In the case of connected sum of tori, the base case is verified by the 12-vertex flag triangulation (whose vertex links are all 6-cycles) that we found in Section \ref{Section 3}. Inductively, assume that there is a flag triangulation $\Delta_1$ of the connected sum of $k-1$ tori with $4+4k$ vertices and furthermore, that there is a vertex $v_1$ of degree 6 in $\Delta_1$. Let $\Delta_2$ be the 12-vertex flag triangulation of $\Sp^1\times \Sp^1$. By the construction, there is a vertex $v_2$ of degree 6 in $\Delta_2$. We let $W_i=V(\st(v_i, \Delta_i))$ for $i=1,2$. Then $\Delta_1[W_1]=\st(v_1,\Delta_1)$ and $\Delta_2[W_2]=\st(v_2, \Delta_2)$ are indeed simplicial 2-balls whose boundaries are both 6-cycles. Since the vertex stars and vertex links are induced, by Lemma \ref{lm: flag connected sum} the flag connected sum $\Delta=\Delta_1\#\Delta_2$ is well-defined and triangulates the connected sum of $k$ tori. Furthermore,$$
	f_0(\Delta)=f_0(\Delta_1)+f_0(\Delta_2)-f_0(\st(v_1, \Delta_1))-1=(4+4k)+12-7-1=8+4k.$$ Finally, to make the inductive construction work, we need to check that there is a vertex $w$ of degree 6 in the new complex $\Delta$. Indeed, we can take any vertex $w$ from $\Delta_2$ such that $w\notin \st(v_2, \Delta_2)$. Any such $w$ has $\st(w,\Delta_2)=\st(w,\Delta)$ and so it is of degree 6 in $\Delta$.
	
	The proof for connected sum of $\R P^2$ is similar. We begin with an 11-vertex flag triangulation of $\R P^2$ as in Figure \ref{figure: RP^2} and note that it has two disjoint vertices of degree 6 (for example, the vertices 1 and 11). Inductively we take the flag connected sum of a $(5+3k)$-vertex flag triangulation $\Gamma_1$ of the connected sum of $k-1$ copies of $\R P^2$ found by induction with an 11-vertex flag triangulation $\Gamma_2$ of $\R P^2$. This is done by choosing a vertex $u$ of degree 6 in $\Gamma_1$ (whose existence is by induction), removing all interior faces of $\st(u, \Gamma_1)$ and $\st(11, \Gamma_2)$, and identifying $\lk(u, \Gamma_1)$ with $\lk(11, \Gamma_2)$. The vertex 1 from $\Gamma_2$ is also a vertex of degree 6 in $\Gamma_1\#\Gamma_2$. Finally, $$f_0(\Gamma_1\#\Gamma_2)=f_0(\Gamma_1)+f_0(\Gamma_2)-f_0(\st(u,\Gamma_1))-1=(5+3k)+11-7-1=8+3k.$$
\end{proof}

\begin{remark}
	As shown in Section \ref{Section 3}, the computer search found 28 combinatorially distinct 14-vertex flag triangulations of the Klein bottle. Indeed many of these 14-vertex triangulations can be obtained by taking the flag connected sum of two 11-vertex flag triangulation of $\R P^2$, as suggested in the proof of Theorem \ref{thm: small flag triangulations of all surfaces}.
\end{remark}
\section{Towards a manifold Charney-Davis conjecture}\label{Section 5}
As mentioned in Section \ref{Section 1}, while the lower bound theorem for flag 3-spheres is established by Davis and Okun, so far there is no analogous lower bound conjecture for flag 3-manifolds. To motivate why such a conjecture might exist, let us recall the classical lower bound theorems for manifolds and balanced manifolds. A simplicial $(d-1)$-sphere is \emph{stacked} if it is the connected sum of the boundaries of $d$-simplices. The \emph{Walkup class} $\mathcal H^d$ ($d\geq 3$) is defined recursively as follows: 1) $\mathcal{H}^d(0)$ is the set of all stacked $(d-1)$-spheres, that is, the spheres obtained by successively taking the connected sum of the boundary complexes of the simplices, 2) a simplicial complex $\Delta$ is in $\mathcal{H}^d(k + 1)$ if it is obtained from a member of $\mathcal{H}^d(k)$ by a
handle addition, and 3) $\mathcal{H}^d=\cup_{k\geq 0}\mathcal{H}^d(k)$. 
\begin{theorem}[\cite{Datta-Murai},\cite{Murai},\cite{Novik-Swartz}]\label{LBT-manifolds}
	For any connected simplicial $(d-1)$-manifold without boundary $\Delta$ and $d\geq 4$, $$g_2(\Delta):=f_1(\Delta)-df_0(\Delta)+\binom{d+1}{2}\geq \binom{d+1}{2}\beta_1(\Delta).$$ Furthermore, the equality is attained if and only if $\Delta$ is in the Walkup class.
\end{theorem}

We remark that a weaker inequality $g_2\geq 0$ was first proved in the class of simplicial polytopes \cite{Barnette} and in the class of simplicial manifolds \cite{Barnette2}. Using rigidity
theory of frameworks, Kalai \cite{Kalai} found an alternative proof of the inequality and also characterized all simplicial $(d-1)$-manifolds with $g_2=0$ assuming $d \geq 4$. In recent years, lower bound results were also established for balanced simplicial polytopes and manifolds. A $(d-1)$-dimensional simplicial complex $\Delta$ is called \emph{balanced} if there is a coloring map $\kappa: V(\Delta)\to [d]$ such that if $\{a, b\}\in \Delta$, then $\kappa(a)\neq \kappa(b)$. The balanced Walkup class, $\mathcal{BH}^d$, consists of balanced $(d-1)$-dimensional complexes that are obtained from the boundary complexes of $d$-cross-polytopes by successively applying the operations of balanced connected sums and balanced handle additions. 
\begin{theorem}[\cite{Klee-Novik}, \cite{BLBT}]\label{BLBT-manifolds}
	For any connected balanced simplicial $(d-1)$-manifold without boundary $\Delta$ and $d\geq 4$,  $$\bar{g}_2(\Delta):=2f_1(\Delta)-3(d-1)f_0(\Delta)+2d(d-1)\geq  4\binom{d}{2}\beta_1(\Delta).$$ Furthermore, for $d\geq 5$ the equality holds if and only if $\Delta$ is in the balanced Walkup class.
\end{theorem}  
In particular, the above theorems imply that if $\Delta$ is a simplicial $(d-1)$-manifold (balanced simplicial $(d-1)$-manifold, resp.) that attains the lower bound on $g_2$ ($\bar{g}_2$, resp.), then its geometric realization $\|\Delta\|$ could be (the connected sum of) sphere bundles over the circle but could not be a $(d-1)$-dimensional torus, $\R P^{d-1}$, etc. 

As suggested in \cite{Gal}, the analog of $g$-numbers for any flag $(d-1)$-manifold $\Delta$ are the $\gamma$-numbers. In what follows, we will only consider the first two $\gamma$-numbers of $\Delta$, given by $$\gamma_1(\Delta)=f_0(\Delta)-2d, \quad \gamma_2(\Delta)=f_1(\Delta)-(2d-3)f_0(\Delta)+2d(d-2).$$ The following statements are the celebrated Davis-Okun theorem \cite{DavisOkun-dimension 3 CD conjecture} and a special case of the $\gamma$-conjecture \cite{Gal} (which we will call the $\gamma_2$-conjecture). The $\gamma_2$-conjecture has been verified for several classes of flag simplicial spheres; see, for example, \cite{Karu} and \cite{Nevo-Petersen}.
\begin{theorem}[\cite{DavisOkun-dimension 3 CD conjecture}]
	Let $\Delta$ be a flag simplicial 3-sphere. Then $\gamma_2(\Delta)\geq 0$.
\end{theorem}
\begin{conjecture}[\cite{Gal}]
	Let $\Delta$ be a flag simplicial $(d-1)$-sphere. Then $\gamma_2(\Delta)\geq 0$.
\end{conjecture}
As we see from Theorems \ref{LBT-manifolds} and \ref{BLBT-manifolds}, the lower bound theorem for (balanced) simplicial spheres can be extended to (balanced) simplicial manifolds. A natural question to ask is: does there also exist a generalization of the Davis-Okun theorem to the class of flag simplicial 3-manifolds? Is there a reasonable manifold $\gamma_2$-conjecture?

\subsection{Generating flag triangulations of 3-manifolds}
In this subsection, we collect data on certain flag 3-manifolds $\Delta$ to test whether there is relation between the minimum $\gamma_2(\Delta)$ and $\beta_1(\Delta)$. The candidates of types of manifolds that could attain small $\gamma_2$ with respect to $\beta_1$ are (the connected sum) of sphere bundles over the circle. To obtain a flag triangulation of a 3-manifold $M$, one can always take the barycentric subdivision of any triangulation of $M$. However,  taking barycentric subdivisions usually generates a lot of new vertices. In what follows we introduce another useful way to construct flag triangulations of products of manifolds, and discuss how to apply the flag connected sum. This applies to finding small flag triangulations of $\#_i \Sp^2\times \Sp^1$,  $\#_i \Sp^2\twprod \Sp^1$ and $\#_i \Sp^1\times \Sp^1\times \Sp^1$.
 
{\textbf{Step 1: Construct the flag triangulations of products.}

The method of constructing triangulations of products is well known; see, for example, \cite{Lee} and \cite{Eilenberg-Steenrod}. We follow the description in \cite[Section 3]{Lutz}. Let $\Delta_1$ and $\Delta_2$ be the flag triangulations of two manifolds $M_1$ and $M_2$, respectively. A cell decomposition of $M_1\times M_2$ is given by taking the union of all cells $\sigma_m\times \sigma_n$, where $\sigma_m=\{u_0, u_1, \dots, u_m\}$ is an $m$-face of $\Delta_1$ and $\sigma_n=\{v_0, v_1, \dots, v_n\}$ is an $n$-face of $\Delta_2$. Identify the vertices $\{(u_i, v_j): 0\leq i\leq m,\; 0\leq j \leq n\}$ with the points $\{(i,j): 0\leq i\leq m,\;0\leq i\leq n\}$ in $\Z^2$. Then the set of all monotone increasing lattice paths $(u_{i_0}=u_0, v_{i_0}=v_0)-(u_{i_1}, v_{i_1})-\dots -(u_{i_{m+n}}=u_m, v_{i_{m+n}}=v_n)$ corresponds to the set of facets $\{(u_{i_0}, v_{i_0}), \dots, (u_{i_{m+n}}, v_{i_{m+n}})\}$ of $\sigma_m\times \sigma_n$. This is called the \emph{staircase triangulation of product of simplices}. Furthermore, if the vertices of $\Delta_1$ and $\Delta_2$ are totally ordered, then the union of all faces of the staircase triangulation of $\sigma_m\times \sigma_n$ gives a \emph{consistent} product triangulation of $M_1\times M_2$. We call it the \emph{staircase triangulation of $\Delta_1\times \Delta_2$}.
\begin{lemma}
	If $\Delta_1$ and $\Delta_2$ are flag triangulations of the manifolds $M_1$ and $M_2$, respectively, then the staircase triangulation $\Delta$ of $\Delta_1\times \Delta_2$ is also flag.
\end{lemma}
\begin{proof}
	Assume that the vertices of $\Delta_1$ and $\Delta_2$ are totally ordered as $(u_0, u_1, \dots)$ and $(v_0, v_1, \dots)$, respectively. Suppose  $F$ is a missing face of $\Delta$ of size $k > 2$. 
	
	First we claim that there is an order of $V(F) = \{s_1,s_2,\dots,s_k\}$ where $s_i = (u_{p_i}, v_{q_i})$ such that $p_1 \leq p_2 \leq \dots \leq p_k$ and $q_1 \leq q_2 \leq \dots \leq q_k$. Indeed, if for some $i<j$ we have $p_i<p_j$ and $q_i>q_j$ (or similarly, $p_i>p_j$ and $q_i<q_j$), then the points $(p_i, q_i)$ and $(p_j, q_j)$ do not belong to any monotone increasing lattice path in $\Z^2$, and hence $\{s_i, s_j\}\notin \Delta$, a contradiction.
	
	By the definition of the staircase triangulation, any pair of distinct $u_{p_i}$ forms an edge in $\Delta_1$. Hence, $P = \{u_{p_1}, u_{p_2}\dots u_{p_k}\}$ is a clique in $\Delta_1$. Since $\Delta_1$ is flag, we have $P \in \Delta_1$. Similarly, $Q = \{v_{q_1}, v_{q_2}, \dots, v_{q_k}\}\in \Delta_2$. In other words, all the vertices of $F$ are in the cell $P \times Q$. Since the vertices of $F$ are in an increasing order in both $u$ and $v$, they are also lattice points on some monotone increasing lattice path from $(p_1, q_1)$ to $(p_k, q_k)$. Hence $F$ is a face in the staircase triangulation of $P \times Q$, i.e., $F \in \Delta$, which leads to a contradiction.
\end{proof}

As an illustration, to form flag triangulations of $\Sp^2\times \Sp^1$ (orientable) and $\Sp^2\twprod \Sp^1$ (nonorientable), we take the flag triangulation of $\Sp^2$ as the octahedral 2-sphere and the flag triangulation of $\Sp^1$ as the 4-cycle. Based on the above lemma, we generate two triangulations as shown in Figure \ref{fig: flag triangulation of sphere bundle}.
\begin{figure}
	\begin{multicols}{2}
	\centering
	\includegraphics[scale=0.8]{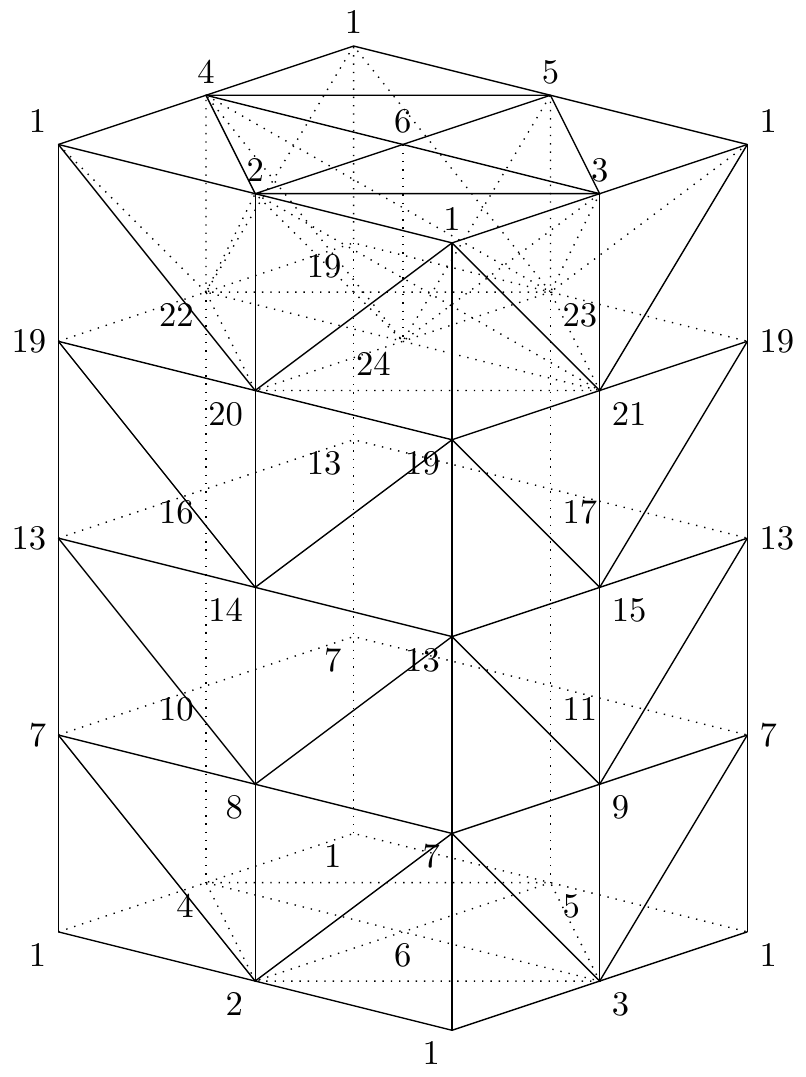}
	\includegraphics[scale=0.8]{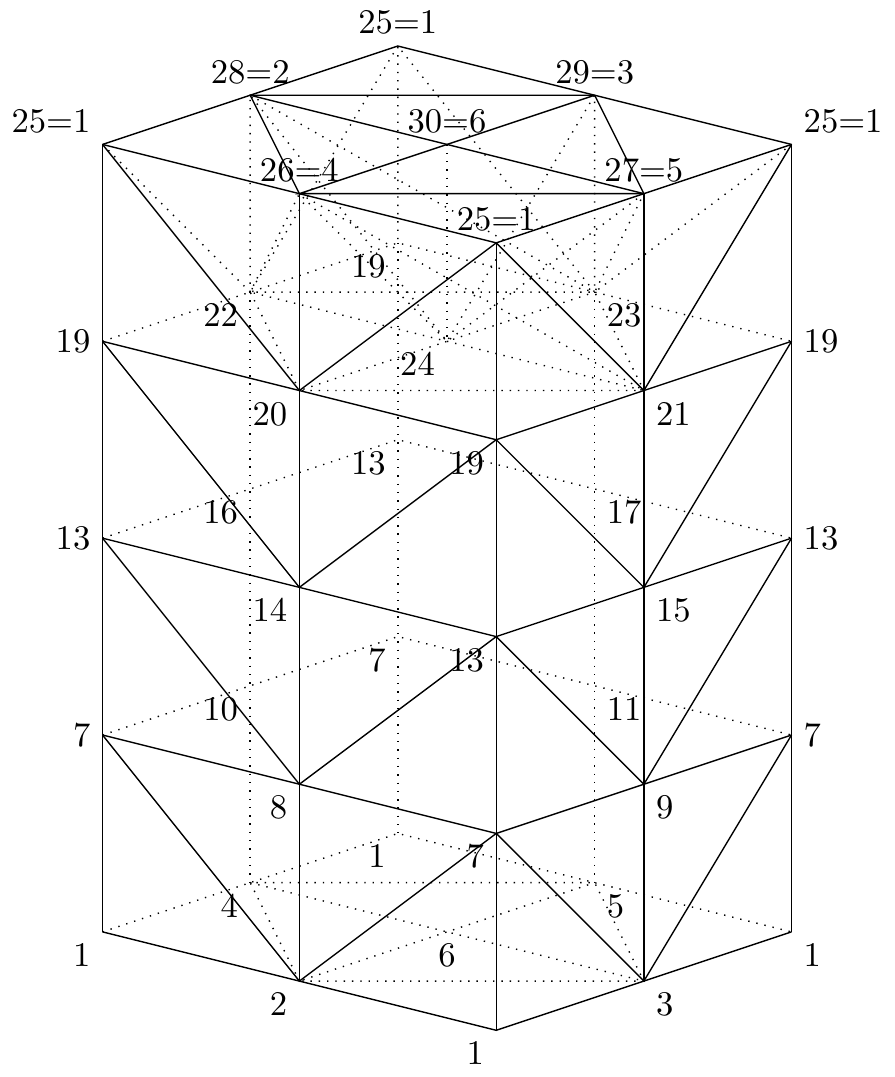}
	\end{multicols}
\caption{A triangulation of the orientable and non-orientable 3-dimensional sphere bundles over $\Sp^1$, respectively. Only the edges in the top layer of the prism are shown for simplicity.}\label{fig: flag triangulation of sphere bundle}
\end{figure}

{\textbf{Step 2: Apply the flag connected sum.}} 

To generate the flag connected sum of flag 3-manifolds, we identify isomorphic edge stars and then remove all the faces containing the corresponding edges. (We use the fact that the edge stars are induced subcomplexes in the flag complexes whose boundaries are also flag.) One advantage of applying this special flag connected sum is that it does little change to the graphs of the complexes.

{\textbf{Step 3: Search for minimal $\gamma_2$ by the Lutz-Nevo Theorem.}}

We generate a few flag triangulations of 3-manifolds by using either the staircase triangulation or the barycentric subdivision of the minimal (non-flag) triangulation. From the Lutz-Nevo theorem and our implementation as described in Section \ref{Section 3}, we obtain data on estimated minimum $\gamma_2$ for various triangulated 3-manifolds, summarized in Table \ref{table: gamma_2 of flag 3-manifolds}.

\begin{table}
	\caption{Estimated minimal $\gamma_2$ for several flag 3- and 4-manifolds; here $[n]=\{1,2,\dots, n\}$}\label{table: gamma_2 of flag 3-manifolds}
	\begin{center}
		\begin{tabular}{ c c c}\hline\hline
			& $\beta_1$ & minimum $\gamma_2$ \\ \hline
			$\Sp^3$ & 0 & 0 \\ \hline
			$\R P^3$ & 0 & 38 \\\hline
			$L(3,1)$ & 0 & 82 \\ \hline
			$\#_i(\Sp^2\times \Sp^1)$, for $i\in [10]$ & $i$ & $16i$ \\ \hline
			$\#_i(\Sp^2\twprod\Sp^1)$ for $i\in [10]$ & $i$ & $16i$ \\ \hline
			$\#_i(\Sp^1\times\Sp^1\times \Sp^1)$ for $i=1,2$ & $3i$ & $112i$\\ \hline
			$\#_i \Sp^3\times \Sp^1$, for $i\in [4]$	& $i$ & $30i$ \\ \hline
			\hline
		\end{tabular}
	\end{center}
\end{table}

We propose the following 3-dimensional manifold Charney-Davis conjecture.

\begin{conjecture}
	Let $\Delta$ be a flag 3-manifold. Then $\gamma_2(\Delta)\geq 16\beta_1(\Delta)$. 
\end{conjecture}

Several remarks are in order. First, it is not surprising that the estimated minimum $\gamma_2$ could be achieved by many distinct flag triangulations of a given manifold. Second, the data suggest that the minimum $\gamma_2$-numbers for flag triangulations of certain 3-manifolds are highly linear with respect to their first Betti numbers. Third, among the 3-manifolds having the same $\beta_1$, the connected sum of sphere bundles over the circle has smaller $\gamma_2$. Fourth, although the lower bound on $\gamma_2$ for higher dimensional flag manifolds is also of interest, our program is not efficient enough to get reliable estimation in these cases. For sake of completeness, we include the results on the connected sum of $\Sp^3\times \Sp^1$ in Table \ref{table: gamma_2 of flag 3-manifolds}. (It is possible that there is a more general conjecture $\gamma_2\geq 2d(d-2)\beta_1$. For lack of evidence, we exclude it from our conjecture.)

We close by showing the lower bound on $\gamma_2$ in the conjecture, if true, is tight. First we define the handle addition on flag complexes.
\begin{definition}
	Let $\Delta$ be a pure flag simplicial complex of dimension $d-1$. For
any two vertices $u$ and $v$, let $\dist(u,v)$ be the length of a shortest path
from $u$ to $v$ in $\Delta$. Assume that there are two disjoint subsets $W_1,
W_2\subset V(\Delta)$ such that 1) both $\Delta[W_1]$ and $\Delta[W_2]$ are
simplicial $(d-1)$-balls, 2) $\partial \Delta[W_1]$ and $\partial \Delta[W_2]$
are isomorphic flag simplicial $(d-2)$-spheres, where $\phi: \partial
\Delta[W_1]\to \partial\Delta[W_2]$ is a simplicial isomorphism, and 3)
$\dist(u,v)\geq 4$ for every $u\in W_1$ and $v\in W_2$. The simplicial complex
$\Delta^\phi$ obtained from $\Delta$ by removing all interior faces of
$\Delta[W_1]$ and $\Delta[W_2]$ and identifying each $v\in W_1$ with $\phi(v)$
is called a \emph{flag handle addition} to $\Delta$.   
\end{definition}
\begin{lemma}
	If $\Delta$ is flag, then a flag handle addition $\Delta^\phi$ is flag.
\end{lemma}
\begin{proof}
	Let $V=V(\Delta)$ and $\bar{\phi}: \Delta\to \Delta^\phi$ be the handle addition map. Assume that $F$ is a missing face of $\Delta^\phi$ of size $>2$. If $\bar{\phi}^{-1}(F)\in \Delta$ contains a pair of vertices $u\in W_1$ and $v\in W_2$, it also contains a path from $u$ to $v$ in $\Delta$. However, $\dist(u,v)\geq 4$ and hence $F$ cannot be the clique on $V(F)$. Now assume that $V(F)\subseteq V\backslash W_2$. Since $\Delta[V\backslash W_2]$ is flag, we have $F\in \Delta[V\backslash W_2]$. Furthermore, $\Delta^\phi[W_1]=\partial \Delta[W_1]$ is flag implies that $F\in\Delta^\phi$, which leads to a contradiction.
\end{proof}
\begin{lemma}\label{lm: gamma2 identity}
	Let $\Delta_1$ and $\Delta_2$ be two flag 3-manifolds. Further assume that for two edges $e_1\in\Delta_1$ and $e_2\in\Delta_2$, there exists a simplicial isomorphism $\phi: \partial \st(e_1, \Delta_1)\to \partial\st(e_2, \Delta_2)$. Then $$\gamma_2(\Delta_1\#_\phi \Delta_2)=\gamma_2(\Delta_1)+\gamma_2(\Delta_2)+2\gamma_1(\lk(e_1, \Delta_1)).$$
  Similarly, if $\Delta$ is a flag 3-manifold containing two edges $\sigma_1, \sigma_2$ such that there is a simplicial isomorphism $\psi: \partial\st(\sigma_1, \Delta)\to \partial\st(\sigma_2, \Delta)$ that defines a flag handle addition on $\Delta$, then 
$$\gamma_2(\Delta^\psi)=\gamma_2(\Delta)+2\gamma_1(\lk(\sigma_1,\Delta))+16.$$
\end{lemma}
\begin{proof}
	By the definition of $\gamma_2$, 
	$$\gamma_2(\Delta_1\#_\phi \Delta_2) = f_1(\Delta_1\#_\phi \Delta_2)-5f_0(\Delta_1\#_\phi\Delta_2)+16.$$
Also by the definition of the flag connected sum,
$$f_1(\Delta_1\#_\phi \Delta_2)=f_1(\Delta_1)+f_1(\Delta_2)-f_1(\st(e_1, \Delta_1))-1,\quad f_0(\Delta_1\#_\phi \Delta_2)=f_0(\Delta_1)+f_0(\Delta_2)-f_0(\st(e_1,\Delta_1)).$$
Combining the above equations, we have $$\gamma_2(\Delta_1\#_\phi \Delta_2)=\gamma_2(\Delta_1)+\gamma_2(\Delta_2)+(-f_1(\st(e_1,\Delta_1))+5f_0(\st(e_1, \Delta_1)-17).$$
The last term on the right-hand side of the above equation is
	\begin{equation*}
	\begin{split}
-f_1(\st(e_1))+5f_0(\st(e_1))-17&=-(f_0(\lk(e_1))+1)+2f_0(\lk(e_1))+5(f_0(\lk(e_1))+2)-17\\
&=2f_0(\lk(e_1))-8=2\gamma_1(\lk(e_1)),
\end{split}
\end{equation*}
which proves the first claim. The proof of the second claim is similar:
\begin{equation*}
\begin{split}
\gamma_2(\Delta^\psi)&=(f_1(\Delta)-f_1(\st(e_1))-1)-5(f_0(\Delta)-f_0(\st(e_1)))+16\\
&=\gamma_2(\Delta)-(3f_0(\lk(e_1))+1)-1+5(f_0(\lk(e_1))+2)\\
&=\gamma_2(\Delta)+2\gamma_1(\lk(e_1))+16.
\end{split}
\end{equation*}
\end{proof}

\begin{proposition}
	For every positive integer $b$, there exists a flag 3-manifold $\Delta$ with $\beta_1(\Delta)=b$ and $\gamma_2(\Delta)=16 b$.
\end{proposition}
\begin{proof}
	The construction is done in three steps.
	
	{\bf Step 1: Construct a 3-manifold $\Delta_4$ with two edges $e$ and $e'$ such that 1) their links are 4-cycles, and 2) the vertex sets of their stars are disjoint.}

	Let $\Gamma_1, \Gamma_2, \Gamma_3, \Gamma_4$ be four disjoint copies of
the octahedral 3-sphere, and let $\Delta_1=\Gamma_1$. We will inductively define
a complex $\Delta_i$ for $2\leq i\leq 4$ by
$\Delta_i=\Delta_{i-1}\#_{\phi_{i-1}}\Gamma_i$ as follows. 
Assign colors 1,2,3,4 to the vertices of each $\Gamma_i$ by giving antipodal
vertices the same color. Let $e$ and $e_1$ be distinct edges of color 1,2 in
$\Delta_1$, and let $e_1'$ be an edge of color 1,2 in $\Gamma_2$. There is a
color-preserving isomorphism
$\phi_1:\partial \st(e_1,\Delta_1)\to\partial\st(e_1',\Gamma_2)$, and we define
$\Delta_2=\Delta_1\#_{\phi_1}\Gamma_2$, which may be viewed as the join of the
4-cycle $\lk(e_1',\Gamma_2)$ colored by 3,4 and a 6-cycle colored by 1,2. 

Next choose $e_2=\{v_1,v_2\}\in\Delta_2$, where $v_1$ is a vertex of color 3 in
$\lk(e_1',\Gamma_2)$ and $v_2$ is a vertex of color 2 not in $\Delta_1$. Given
an edge $e_2'$ of color 2,3 in $\Gamma_3$, use a color-preserving isomorphism
$\phi_2:\partial \st(e_2,\Delta_2)\to\partial\st(e_2',\Gamma_3)$ to define
$\Delta_3=\Delta_2\#_{\phi_2}\Gamma_3$. Here we have that
$\lk(e_2',\Delta_3)$ is a cycle of colors 1,4.

Then choose the edge $e_3=\{v_3,v_4\}\in\Delta_3$ such that the vertex
$v_3\in\lk(e_2',\Gamma_3)$ is of color 4 and the vertex $v_4\notin \Delta_2$ is
of color 3. Given an edge $e_3'$ of color 3,4 in $\Gamma_4$, let $e'$ be its
antipodal edge in $\Gamma_4$, and use a color-preserving isomorphism
$\phi_3:\partial \st(e_3,\Delta_3)\to\partial\st(e_3',\Gamma_4)$ to define
$\Delta_4=\Delta_3\#_{\phi_3}\Gamma_4$. Then $e'$ is disjoint from $\Delta_1$
and does note share a color with $e$, thus we verify that the links of
$e$ and $e'$ in $\Delta_4$ are indeed disjoint.

	{\bf Step 2: Construct a flag 3-manifold $\Delta_{16}$ from $\Delta_4$ and apply flag handle addition.}
	
	Take four copies $\Delta_4^1, \Delta_4^2, \Delta_4^3, \Delta_4^4$ of the above $\Delta_4$ to form a flag connected sum $$\Delta_{16} =\Delta_4^1\#_{\psi_1}\Delta_4^2\#_{\psi_2}\Delta_4^3\#_{\psi_3}\Delta_4^4,$$
	where $\psi_i: \partial\st(e', \Delta_4^i)\to \partial\st(e, \Delta_4^{i+1})$ (in this step we forget the colors on the vertices). If $v\in \st(e, \Delta^1_4)$ and $v'\in \st(e', \Delta_4^4)$, then since any path from $v$ to $v'$ must pass one vertex from the identified stars $\st(e', \Delta_4^i)$ ($i=1,2,3$), it follows that $\dist(v, v')\geq 4$. Hence by identifying $\st(e, \Delta_4^1)$ with $\st(e', \Delta_4^4)$ and removing $e, e'$ in a flag handle addition, we obtain a new flag 3-manifold $\Gamma$ with $\beta_1(\Gamma)=1$. Both $\lk(e, \Delta^1_4)$ and $\lk(e', \Delta^4_4)$ are 4-cycles. Hence by Lemma \ref{lm: gamma2 identity}, $$\gamma_2(\Gamma)=\gamma_2(\Delta_{16})+2\gamma_1(\lk(e, \Delta^1_4))+16=16.$$ 
	
	{\bf Step 3: Generate a flag 3-manifold with arbitrary $\beta_1$.}
	
	This is done by taking the flag connected sum of $b$ copies of $\Gamma$ along the stars of edges whose links are 4-cycles. The resulting complex has $\gamma_2=16b$ and $\beta_1=b$. 
\end{proof}
\section{Acknowledgements}
The research was part of Lab of Geometry at Michigan projects offered by the Department of Mathematics at University of Michigan during the winter semester of 2019. We would like to thank Harrison Bray and many others who coordinated the projects. We also thank the referee for valuable feedback.

\end{document}